\documentclass[12pt]{amsproc}
\usepackage{amsmath}
\usepackage{amssymb}
\newtheorem{thm}{Theorem}[section]

\newtheorem{lem}[thm]{Lemma}
\newtheorem{prop}[thm]{Proposition}
\newtheorem{cor}[thm]{Corollary}

\newtheorem{ex}[thm]{Example}
\newtheorem{rem}[thm]{Remark}

\newtheorem{con}[thm]{Construction}
\theoremstyle{definition}
\newtheorem{defn}[thm]{Definition}
\theoremstyle{remark}

\numberwithin{equation}{section}
\newcommand{\sbq}{\subseteq}
\newcommand{\spq}{\supseteq}
\newcommand{\mc}{\mathcal}
\newcommand{\vide}{\emptyset} 
\newcommand{\tbf}{\textbf}
\newcommand{\mbf}{\mathbf}
\newcommand{\mf}{\mathfrak}
\newcommand{\inv}{^{-1}}
\newcommand{\bz}{\mbf{0}}
\newcommand{\C}{\text{C}}
\DeclareMathOperator{\ann}{ann}

\newcommand{\R}{\mathbb R}
\newcommand{\N}{\mathbb N}

\newcommand{\B}{\mathbf B}
\newcommand{\res}{\raisebox{-.5ex}{$|$}}
\newcommand{\cx}{\text{C}(X)}
\newcommand{\spec}{\text{Spec\,}}
\newcommand{\minspec}{\text{Min}\,}
\newcommand{\maxspec}{\text{Max}\,}

\newcommand{\qcl}{Q_\text{cl}}
\newcommand{\epar}{\\[-.5ex]}
\newcommand{\igr}[1]{\text{IG}_{#1}}
\newcommand{\alg}[1]{\text{Alg}_{#1}}

\begin{document}       

 \title[Adjoining idempotents]{Adjoining Idempotents to a Commutative Ring\\ preprint version}

 \author{W.D. Burgess}
\address{Department of Mathematics and Statistics\\ University of Ottawa, Ottawa, Canada, K1N 6N5}

\email{wburgess@uottawa.ca}
\thanks{ Thanks to Concordia University for a grant facilitating the research.  Thanks also to the anonymous referee(s) for the very thorough reading and valuable, insightful comments.}


\author{R. Raphael}
\address{Department of Mathematics and Statistics\\ Concordia University, Montr\'eal, Canada, H4B 1R6}

\email{r.raphael@concordia.ca}

\begin{abstract} Everything takes place in the category of commutative unitary rings. For a fixed ring $R$, $\alg{R}$ is the class of $R$-algebras and $\igr{R}$ the subclass of idempotent generated $R$-algebras. Following Bezhanishvili \emph{et al} and their study of Specker and locally Specker $R$-algebras, this paper studies the interplay of properties of $R$ and $A\in \igr{R}$ (both as rings and as $R$-modules). Examples: (1)~If $R\sbq A\in \igr{R}$ and $R$ is weak Baer (aka p.p.\ ring) and $A$ is ring essential over $R$, then $A$ is weak Baer and locally Specker. (2)~If $R$ is semiprime and all the idempotents of the complete ring of quotients are adjoined to $R$ to form $A$, then $A_R$ is flat iff $R$ is weak Baer, in which case $A$ is locally Specker.

The Pierce sheaf is often used since it is based on idempotents. Properties are examined, old and new, that are true for $R$ iff they are true for all the Pierce stalks. Among the new is the result for f-rings (pure ideals are generated by idempotents): $R$ is an f-ring iff each of its Pierce stalks has no non-trivial pure ideals. This allows the expansion of the known classes of f-rings; f-rings play important roles in $\igr{R}$.  
 \end{abstract}
\maketitle

\setcounter{section}{0} \setcounter{thm}{0}  \tbf{0. Introduction.} Everything that follows takes place in the category $\mc{CR}$ of unitary commutative rings.  The terminology used is that of the basic texts \cite{Ma} and \cite{L}. Most often a ring $R\in \mc{CR}$ is fixed and the point of view is of the category of $R$-algebras, denoted $\alg{R}$. For $A\in \alg{R}$, there is a ring homomorphism $\phi\colon R\to A$.

The focus is frequently on the subcategory, $\igr{R}$, of $\alg{R}$ where $A\in \igr{R}$ is generated, as a ring, by $\phi(R)$ and a set of idempotents.  These idempotents may be from a larger ring or constructed via polynomials modulo an ideal.

There is a long tradition of looking at ring properties that are, or are not, preserved by various constructions such as  polynomial rings, power series rings, products, and so on. This study continues that line of thinking by being about 
the interplay of properties of $R$ and those of $A\in \igr{R}$, both as rings and as $R$-modules.  But, since idempotents are crucial, there is a major  excursion into Pierce sheaves, both as a tool and in their own right.   

The papers by Bezhanishvili \emph{et al}, \cite{BMMO}, and \cite{BMO} form a starting point.  Those papers take a categorical point of view and the latter also studies the Pierce sheaf of $A\in \igr{R}$. 

Before giving an outline of the paper, terminology is established. \epar

\noindent \tbf{Some Notation and Terminology:}

$\bullet$ The boolean algebra of idempotents in a ring $R$ is denoted by $\B(R)$. The zero ideal or submodule is written as $\bz$.

$\bullet$ If $X$ is a subset of a ring $R$, the ideal of $R$ generated by $X$ is denoted $\langle X\rangle$.  Also, if $R\sbq T$ are rings, and $Y$ is a subset of $T$, then the ring generated by $R$ and $Y$ is denoted by $R[Y]$. 

$\bullet$  As usual, the set of prime ideals of a ring $R$ is denoted $\spec R$. It is topologized with the Zariski topology.
 The set of minimal prime ideals of $R$  is denoted $\minspec R$ and the set of maximal ideals, $\maxspec R$.

$\bullet$  The term \emph{regular ring} here always refers to a \emph{von Neumann regular ring}, that is, a ring $R$ such that for all $r \in R$, there is $r'\in R$ with $r^2r' =r$. Note that, in the above equation $rr'$ is an idempotent. A regular ring $R$ is characterized by the fact every $R$-module is flat (e.g., \cite[Theorem~4.21]{L}); and by the fact that $R$ is both semiprime and every prime ideal is maximal (\cite[Theorem~3.71]{L}). 

$\bullet$ The complete (or maximal) ring of quotients (see, for example, \cite{Ste}) of $R$ is written $Q(R)$ and the classical ring of quotients (or total ring of fractions) is $\qcl(R)$. 

$\bullet$ For a topological space $X$, $\cx$ refers to the ring of continuous real-valued functions. As in the classic text \cite{GJ}, the space $X$ may be taken to be completely regular (\cite[\S~3.1]{GJ}). The complete ring of quotients and the classical ring of quotients of $\cx$ are denoted $Q(X)$ and $\qcl(X)$, respectively. 

$\bullet$ (i)~A ring $R$ is called \emph{Baer} or \emph{a Baer ring} if for any subset $Y$ of $R$, $\ann_R Y = eR$ for some $e\in \B(R)$. (ii)~A ring $R$ is called \emph{weak Baer} or \emph{wB} if for any $a\in R$, $\ann_Ra=eR$, for some $e\in \B(R)$. (A  weak Baer ring is also called a \emph{p.p.\ ring}.) (iii)~A ring $R$ is called \emph{almost weak Baer} or \emph{awB} if for $a\in R$, $\ann_Ra = ER$, for some $E\sbq \B(R)$. (The original definition of awB, \cite[page~174]{NR}, is readily seen to be equivalent to the above.)

$\bullet$ An ideal $I$ of $R$ is called \emph{pure} if $R/I$ is a flat $R$-module. A characterization used below is: $I$ is pure if for each $a\in I$ there is $b\in I$ with $ab=a$ (see e.g., \cite[Proposition~1.5]{DeM}). From this, it is clear that every idempotent generated ideal $I$ is pure.  

If every pure ideal of $R$ is generated by idempotents, then $R$ is called an \emph{f-ring} (from \cite{V}). 
\epar

\noindent \tbf{Outline of the paper}.  References are to items later in the paper.

\tbf{Section~1} is devoted to some preliminary notions and results that will be used throughout. There are also formal definitions of Specker and locally Specker algebras from \cite{BMMO} and \cite{BMO}, respectively. 

\tbf{Section~2} is an excursion into Pierce sheaves.  A property $\mc{P}$ of rings is called a \emph{ring Pierce property} if $\mc{P}$ is true for a ring $R$ if and only if it is true for each of the stalks of the Pierce sheaf. Some known ring Pierce properties are listed and some are new. In particular the property of being an f-ring is shown to be a ring Pierce property (\ref{fP'}). This allows an expansion of the known classes of rings that are f-rings, including to awB rings, and to some constructions related to products.  Moreover, f-rings play an important role in $\igr{R}$, as is seen in \cite[Section~7]{BMO}.  Other ring Pierce properties and properties that turn out not to be ring Pierce properties are considered; not all show up in later sections, but they are important in the study of commutative rings. 

\tbf{Section~3} has more on the study of f-rings in view of the material in Section~2.  This includes another, shorter, proof of \cite[Theorem~7.9]{BMO} using that being an f-ring is a ring Pierce property. Also there is a new characterization of clean rings (\ref{pmf=c}). 

\tbf{Section~4}: Algebras in $\igr{R}$ are integral extensions of $R$ and hence have going-up with respect to $R$.  This section looks at a strictly weaker form of going-down (\emph{weak going-down}, Definition~\ref{wgd}) in the context of $\igr{R}$.  It is shown (\ref{gdP}) that when $R$ is semiprime and $R\sbq A\in \igr{R}$, if $A$ has weak going-down with respect to $R$, then $A_R$ is flat. 

\tbf{Section~5}: This short section asks: if $R$ is semiprime and $R\sbq A \in \igr{R}$, when $A$ is semiprime?  If $A$ has weak going-down, then $A$ is semiprime (\ref{flatE}). 

\tbf{Section~6}: This section looks at $\alg{R}$ where $R$ is a wB ring. If $R$ is wB and $A\in \igr{R}$ with $R\sbq A$, and $A$ is ring essential over  $R$ (i.e., for $0\ne a\in A$, $aA\cap R\ne\bz$),  then $A$ is wB and is a locally Specker $R$-algebra (\ref{epp}). These properties are then looked at in the context of rings of continuous functions of the form $\cx$, including some examples when $X$ is an almost P-space.\epar

There are two basic ways of adding idempotents to a ring $R$ (or, from the point of view of $\igr{R}$, to $\phi(R)$).  One is to find them in a bigger ring.  The other is by direct construction.  This latter is examined in detail in \cite{BMMO} where, for any boolean algebra $B$, there is an extension of a ring $R$ by $B$ where the elements of $B$ are adjoined to $R$ as idempotents to get a ring $R[B]_{Sp}$. This is called the  \emph{free extension by $B$} or, following, \cite{BMMO}, a \emph{Specker $R$-algebra}. The construction is given by \cite[Definition~2.4]{BMMO}. Here $R[B]_{Sp}$ is a free $R$-module. There is a more general point of view in \cite{BMO}, where the idempotents of the base ring $R$ and those of $B$ may have interactions. The resulting construction is of a \emph{locally Specker $R$-algebra} (see Definition~\ref{LSp}). \epar

Since idempotents are the focus of the paper, one natural tool is the \emph{Pierce sheaf} (see \cite{P}, \cite{BS}, and \cite[Chapter~V.2]{J}, and also \cite[Introduction]{BMO}).  For any ring $R$, the set of idempotents of $R$, $\B(R)$ (called Id($R$) in \cite{BMO}), can be considered as a boolean ring or a boolean algebra; the latter is used in the sequel.  The key to the construction are the ideals of $R$ called the \emph{Pierce kernels}; there is one for every $x\in \spec \B(R)$, namely $xR$; for $x\ne y$ in $\spec \B(R)$, $xR$ and $yR$ are comaximal (see Lemma~\ref{mp}(3), below).  The stalks of the Pierce sheaf are $R_x= R/xR$.  Recall that $\spec \B(R)$ is a boolean space, i.e.,  compact, Hausdorff and totally disconnected. The topology on the \emph{espace \'{e}tal\'{e}} (the disjoint union, $\bigcup_{x\in \spec \B(R)} R_x$)  of the sheaf  is generated by, for $r\in R$ and $e\in \B(R)$, $\{r_x\mid e\notin x\}$.  The key result is that the ring of global sections of the sheaf is isomorphic to $R$. There is a parallel theory for $R$-modules. The Pierce sheaf of a locally Specker $R$-algebra is studied in detail in \cite{BMO}.   \epar

\setcounter{section}{1} \setcounter{thm}{0}\noindent \tbf{1. Some preliminaries.} Some essential tools are introduced at this stage. The following is used often, without mention, and is folklore.  

\begin{lem}
\label{folk}  Suppose $A$ is an $R$-algebra with $R\sbq A$ and $A=R[B]$ where $B$ is a boolean subalgebra of $\B(A)$. If $a= \sum_{i=1}^k r_ie_i$ with $r_i\in R$ and $e_i\in B$, then there is an orthogonal set $\{f_1,\dots,f_l\}\sbq B$ and $r_j'\in R$, $j=1,\ldots, l$ such that each $r_j'\in \langle r_1, \ldots, r_k\rangle$ and $a=\sum_{j=1}^lr_j'f_j$.  Moreover, for each $j=1,\ldots, l$, there is $i=1,\ldots,k$ with $f_je_i=f_j$, and $\bigvee_{j=1}^l f_j = \bigvee_{i=1}^ke_i$.  \end{lem}

This is  a good moment to connect the algebras of \cite{BMMO} and \cite{BMO} with the point of view taken in this article.  

\begin{defn} \label{Stdef}  \cite[Definition~2.4]{BMMO} Let $R$ be a ring and $B$ a boolean algebra (with join, meet and negation denoted by $\wedge$, $\vee$ and $\neg$, respectively).  The \emph{Specker $R$-algebra defined by $R$ and $B$} is denoted $R[B]_{Sp}$ and is defined as follows: Consider the set $X=\{x_e\mid e\in B\}$ and the polynomial ring $R[X]$ modulo the ideal $J$ generated by $\{x_{e\wedge f} -x_ex_f, x_{e\vee f} -(x_e+ x_f-x_ex_f), x_{\neg e}-(1- x_e), x_0\}_{\forall e,f\in B}$.  Idempotents in $R[B]_{Sp}$ of the form $x_e+J$ are denoted $y_e$.\end{defn}

Notice that in a Specker $R$-algebra, $R[B]_{Sp}$, there are exactly two elements of $B$ that are identified with elements of $R$, namely, $0, 1\in B$ are identified with $0, 1\in R$, respectively. Moreover, for each $e\in B$, $e\ne 0$, $Re\cong R$ as $R$-modules (\cite[Lemma~2.6]{BMMO}).

When $R$ is indecomposable and $B$ is a boolean algebra, then, as boolean algebras, $\B(R[B]_{Sp}) \cong B$ (\cite[Theorem~3.6]{BMMO}). More generally,  $\B(R[B]_{Sp}) \cong B(R)[B]_{Sp}$, as boolean algebras (\cite[Lemma~3.2]{BMMO}(2)). 

If $A\in \igr{R}$ via $\phi\colon R\to \phi (R)$, then $A$ is a homomorphic image of a Specker $R$-algebra (\cite[Lemma~2.2]{BMMO}). This is made more specific in the following.

\begin{prop} \label{qSt}  If $\phi(R)\sbq T$ are rings and $E\sbq \B(T)$, then $A=R[E]$ is a quo\-tient of a Specker $R$-algebra.  The ker\-nel of $\psi\colon R[\B(A)]_{Sp}$ $\twoheadrightarrow A$ contains all $y_e-e$, where $e\in \B(A)$.\end{prop}

\begin{proof} Put $B=\B(A)$ and call the equality $\iota\colon B \to \B(A)$. According to \cite[Lemma~2.5]{BMMO}, there is an $R$-algebra homomorphism from $\psi\colon R[B]_{Sp} \to A$ such that $\psi\res_B = \iota$. But $A$ is generated over $R$ by $B$.  Hence, $\psi$ is a surjection. 
\end{proof}

Note that, in the above, $\ker \psi$ contains all $y_e-e$, where $e\in \B(A)$, but may have more generators, as can be seen in some of the examples later in this article (e.g., Example~\ref{nsp}).

Also important for this paper is the notion of a \emph{locally Specker algebra} from \cite[Definition~3.8 and Theorem~3.13]{BMO}. Note, the following is not the definition in \cite{BMO}, but is shown there to be equivalent to it (\cite[Theorem~3.13]{BMO}). 

\begin{defn} \label{LSp} Let $R$ be a ring and $B$ a boolean algebra.  Suppose that there is a boolean algebra homomorphism $\gamma\colon \B(R) \to B$. Then the \emph{locally Specker algebra} defined by these data is $R[\{x_b\mid b\in B\}]/J$,  where the ideal $J$ is generated by: $x_{e\wedge f} -x_ex_f$, $x_{e\vee f} - (x_e+x_f-x_ex_f)$, $x_{1-e} -(1-x_e)$, $x_0$ and $x_{\gamma(g)} -gx_1$, for all $e,f\in B$ and all $g\in \B(R)$. The element $x_e+J$ is denoted by $z_e$, for all $e\in B$.   \end{defn}              

A locally Specker $R$-algebra $A$ has a Pierce structure that is fully described in  \cite[Theorem~3.3]{BMO}.  Note that, in this case, $\phi (R) = R/(\ker \gamma) R =\phi(R)$, and $\ker\phi$ is generated by idempotents and, hence, is a pure ideal. Moreover, using \cite[Theorem~3.3(2) and (3)]{BMO}, every Pierce stalk of $\phi(R)$ appears as a Pierce stalk of $A$, and, conversely, every Pierce stalk of $A$ is isomorphic to one of the Pierce stalks of $\phi(R)$.  

A recurrent theme is that under certain conditions on $A\in \igr{R}$, the conclusion will be that $A$ is a locally Specker $R$ algebra.   

A basic result is that if $R$ is any ring and $A \in IG_R$  then each Pierce stalk of $A$ is a homomorphic image of a Pierce stalk of $R$. This is easily proved, via Lemma~\ref{folk}, but is also \cite[Lemma~1.5]{BR}. 
 
\begin{lem}\label{addi}  \cite[Lemma~1.5]{BR}  If $A\in \igr{R}$, then each Pierce stalk of $A$ is the homomorphic image of a Pierce stalk of $R$. \end{lem}

This prop\-erty even char\-ac\-ter\-izes elements of $\igr{R}$, by \cite[Pro\-po\-sition~3.2(3)]{BMO}, and is, thus, a key element in many proofs. 

A corollary of Lemma~\ref{addi} gives the first examples of properties preserved by adding idempotents. Recall that a \emph{clean ring} (see, e.g., \cite[Introduction]{NZ})  is one where each element is a sum of an idempotent and a unit.  A ring is clean if and only if each of its Pierce stalks is local (\cite[Proposition~1.2]{BS}). Moreover, a regular ring is characterized by one whose Pierce stalks are fields. Both of these properties are preserved by homomorphic images. These classes of rings will appear often in this article. 

Before stating the corollary, a lemma will be useful.  

\begin{lem} \label{vNr}  If $R\sbq S$ are regular rings and $y\in \spec \B(S)$ and $x=yS\cap \B(R)$, then $yS \cap R= xR$. \end{lem} 
\begin{proof} As always $xR\sbq yS\cap R$. There is a monomorphism $R/(yS\cap R) \to S/yS$; since the latter is a field, $yS\cap R$ is a prime ideal in $R$, hence, maximal.  Moreover, $xR$ is maximal.  Hence, $yS\cap R=xR$. 
\end{proof}  

\begin{cor} \label{clean}  Let $R$ be a ring and $A \in\igr{R}$.

 (1)~If $R$ is clean then $A$ is clean.

(2)~If $R$ is regular then $A$ is regular.  In fact in (2), each Pierce stalk of $A$ is a copy of a Pierce stalk of $R$ and $A$ is a locally Specker $R$-algebra. 
 
 (3)~If $R$ is regular and $\phi(R)\sbq T \sbq A$, where $T$ is a ring, then $T$ is regular.  In fact, $T$ is a locally Specker $R$-algebra.  \end{cor}

\begin{proof} (1)~By Lemma~\ref{addi}, each Pierce stalk of $A$ is a homomorphic image of a local ring, and is, hence, local. (2)~Similarly, each homomorphic image of a field is the same field. Moreover, \cite[Theorem~3.3(3) and Definition~3.4]{BMO} then says that $A$ is a locally Specker $R$-algebra. (3)~As in (2), $A$ is regular (since $\phi(R)$ is) and integral over $\phi(R)$.  Hence, $T$ is integral over $\phi(R)$. If $\mf{p}\in \spec T$ is not maximal, then there is a maximal $\mf{m}$ with $\mf{p}\subset \mf{m}$. But, $\mf{p}\cap \phi(R)$ is maximal (because in a regular ring all primes are maximal) and so $\mf{p}\cap \phi(R) = \mf{m}\cap \phi(R)$, which is impossible in an  integral extension.  Thus, $\mf{p}$ is maximal. Since $T$ is semiprime, $T$ is regular and the maximal ideals are precisely the Pierce kernels. 

Let $x\in \spec \B(T)$ and $y=x\cap \B(\phi(R))$. There is $z\in \spec \B(A)$ with $z\cap \B(T) =x$. Since $\phi(R)$ is regular, $xT\cap \phi(R)=y\phi(R)$ (Lemma~\ref{vNr} applied to $\phi(R)$ and $T$).  Put $\alpha \colon \phi(R)/y\phi(R)\to T/xT$, an injection, and $\beta\colon T/xT\to A/zA$. Then, $\beta\alpha\colon \phi(R)_y\to A_z$ is an isomorphism since $A$ is a locally Specker $R$-algebra. Since $T_x$ is a field, $\beta$ is a monomorphism, making $\beta$ an isomorphism. This makes $\alpha$ an isomorphism as well.  Again by \cite[Theorem~3.3(3)]{BMO}, $T$ is a locally Specker $R$-algebra.
  \end{proof}

The following facts about Pierce stalks and kernels will be useful.  The first statement in the lemma will be applied to Pierce kernels.

\begin{lem} \label{mp} Let $R$ be a ring and $x\in \spec \B(R)$.\\ (1)~Let $I$ be a pure ideal.  If $\mf{p}\in \spec R$ contains $I$ and if $\mf{q} \in \spec R$ with $\mf{q}\sbq \mf{p}$, then $I\sbq \mf{q}$.  In particular, if $\mf{p}$ is a minimal prime over $I$, $\mf{p}\in \minspec R$. \\  (2)~The Pierce stalk $R_x$ is a classical localization of $R$. Moreover,  $R_x$ is a flat $R$-module.\\  (3)~The Pierce kernels are co-maximal.\\  (4)~For $\mf{p}\in \spec R$, there is exactly one $x\in \spec \B(R)$ where $\mf{p}+xR \ne R$.  \end{lem}

\begin{proof} (1) Here, if $a\in I$ there is $b\in I$ with $a(1-b) = 0$.  Hence, $a\in \mf{q}$ or $1-b \in \mf{q}\sbq \mf{p}$.  The latter means $1\in \mf{p}$, a contradiction. Hence, $a\in \mf{q}$. The second statement of (1) follows.   (2)~The set $S=\B(R)\setminus x$ is a multiplicatively closed set and $R_x\cong RS\inv$. The last statement of (2) follows since all classical localizations are flat (e.g., \cite[Proposition~10.8]{Ste}).   (3)~If $x, y\in \spec \B(R)$, $x\ne y$, there is $e\in y\setminus x$.   Then, $1-e\in x$ showing that $1\in xR+ yR$.  (4)~For $\mf{p}\in \spec R$, it is easy to see that $x_0=\mf{p}\cap \B(R)$ is in $\spec \B(R)$.  By (3), it is the only Pierce kernel in $\mf{p}$.  For $x\ne x_0$ there is $e\in x_0\setminus x$, then $\mf{p}+xR = R$, while $\mf{p}+x_0R = \mf{p}$.  
\end{proof} 

Of course, a Pierce stalk $R_x$ of a ring $R$ is also flat because $R_x = R/xR$, and $xR$ is clearly a pure ideal.  The second part of Lemma~\ref{mp}(1) is \cite[Proposition~1.6]{DeM}.

Since every semiprime ring $R$ can be embedded in a regular Baer ring, the following simple lemma turns out to be useful.

\begin{lem} \label{Betal}   Let $R$ be a ring and $R\sbq T$, where $T$ is a (i)~Baer ring, (ii)~a wB ring, (iii)~an awB ring.  Suppose the rings $R\sbq S\sbq T$ are such that $\B(S) = \B(T)$, then $S$ inherits the properties (i), (ii) or (iii), respectively.  This applies, in particular, if $S = R[\B(T)]$. \end{lem}

\begin{proof} Always, if $X$ is a subset of $S$, $\ann_SX= \ann_TX \cap S$. The lemma follows because each of the three properties of $T$ are defined by annihilators and idempotents and $\B(S) = \B(T)$. 
\end{proof} 

In fact more can be said when $T=Q(R)$.
 
\begin{cor} \label{empp}    Every semiprime ring $R$ can be embedded in a Baer ring $S$ in $\igr{R}$.  If $S$ is a subring with $R\subset S\sbq Q(R)$ and $\B(S) = \B(Q(R))$, then $S$ is a Baer ring. Moreover, $S_R$ is a module essential extension of $R_R$ (i.e., for $0\ne s\in S$, $sR\cap R\ne\bz$). \end{cor} 

\begin{proof} Since $R$ is semiprime, its complete ring of quotients $Q(R)$ is regular and self-injective (e.g., \cite[Proposition~2.1, page~247]{Ste}), hence, $Q(R)$ is a Baer ring. The rest is from Lemma~\ref{Betal}. The last statement follows since $Q(R)_R$ is a module essential extension of $R_R$.
\end{proof} 

\setcounter{section}{2} \setcounter{thm}{0}
\noindent\tbf{2. Some Pierce properties.} This section has two purposes.  The first is to examine properties of rings and how they relate to the Pierce sheaf.  The second is that some of the properties chosen are of interest in later sections and Pierce tools will be used there.  Recall that in the Pierce sheaf for $R$, if an equation holds in a Pierce stalk $R_x$ then it holds over a clopen set containing $x$ in the base space $\spec \B(R)$ (\cite[Lemma~4.3]{P}). If this is true for all $x$, then compactness of $\spec \B(R)$ allows an equation to hold globally.  

Before looking at the main topic of this section, here are two useful constructions of Pierce sheaves from given indecomposable rings. First the direct product and then a variant.

\begin{con} \label{con1}  Let $R_\alpha$, $\alpha \in A$, be a set of  indecomposable rings.  Put $P =\prod_{\alpha \in A}R_\alpha$. The Pierce sheaf of $P$ will be described in the next paragraph.  \end{con}

(This situation is dealt with in \cite[the start of \S3]{Ao}, even though the Pierce sheaf is not mentioned there.) The elements of $\B(P)$ can be identified with the subsets of $A$, making $\spec \B(P)$ homeomorphic to the space $\beta A$, the \v{C}ech-Stone compactification of the discrete space $A$,  that is, the space of ultrafilters on $A$. For a fixed ultrafilter associated with $\alpha\in A$, the Pierce stalk of $P$ is the ring $R_\alpha$. For a free ultrafilter $\mc{U}$, the corresponding Pierce stalk of $P$ is an ultraproduct of the rings $R_\alpha$.  

There are variants of this when the $R_\alpha$ have a  common subring $R$. 

\begin{con} \label{con2}  With $P$ as in Construction~\ref{con1}, suppose that the $R_\alpha$ all have a common subring $R$. 

(1)~Put $$S=\{(r_\alpha)\in P\mid \text{$r_\alpha \in R$ for all but finitely many $\alpha\in A$}\}\;.$$  The Pierce sheaf of $S$ is described.

(2)~Put $$T= \{(r_\alpha)\in P\mid \text{for some $r\in R$, $r_\alpha=r$ for all but finitely many $\alpha$}\}.$$  The Pierce sheaf of $T$ is described.\end{con}

 \begin{proof} (1)~Once again, $\B(S)$ can be identified with the subsets of $A$, that is, $\B(S) = \B(P)$. However, since no finite set is in an ultrafilter, the Pierce stalks of $S$ associated with fixed ultrafilters are again the $R_\alpha$, while those associated with free ultrafilters are ultrapowers of $R$. 

(2)~This is \cite[Proposition~2]{BS}. The underlying space is $A\cup \{\infty\}$,  the one-point compactification of the discrete $A$.  For $\alpha\in A$, the Pierce stalk is $R_\alpha$ and $T_\infty = R$. The function required by the reference $d\colon R\to P/\bigoplus R_\alpha$ is the diagonal followed by dividing out by the direct sum. \end{proof} 

Note that any property of a ring $R$ that is preserved by homomorphic images, goes from $R$ to its Pierce stalks. 

\begin{defn} \label{PP}  (1)~A property $\mc{P}$ of rings is called a \emph{ring Pierce property} if a ring $R$ has the property if and only if each of its Pierce stalks has it.  (2)~A property $\mc{P}$ of $R$-modules is called a \emph{module Pierce property} if a module $M_R$ has the property if and only if, for each $x\in \spec \B(R)$, $(M_x)_{R_x}$ has it. \end{defn}

The next proposition has some of the known ring Pierce properties. Recall that a \emph{pm-ring} is one in which every prime ideal is contained in a unique maximal ideal. Every $\cx$ is such a ring (\cite[14.3(c)]{GJ}).  

\begin{prop} \label{known}    The follo\-wing are ring Pierce prop\-er\-ties: \\(i)~clean ring, (ii)~regular ring, (iii)~awB ring, (iv)~pm-ring, and  (v)~semi\-prime ring. \end{prop} 

\begin{proof} (i)~This is \cite[Proposition~1.2]{BS2}, which says that $R$ is clean if and only if each $R_x$ is local. However, an indecomposable clean ring is a local ring and reciprocally. (ii)~$R$ is regular if and only if each $R_x$ is a field. Again an indecomposable regular ring is exactly a field. (iii)~By \cite[Theorem~2.2]{NR} $R$ is awB if and only if each $R_x$ is a domain. Again, an indecomposable awB ring is a domain and reciprocally. (iv)~\cite[Theorem~4.1]{Co} has an element-wise characterization of pm-rings that makes it clear that this is a ring Pierce property; but see also \cite[Remark~3.1]{BR} for another approach. (v)~Both  directions of ``not semiprime'' are easy.
\end{proof} 

Part (v) of \ref{known} has an immediate corollary.  If $A\in \igr{R}$ is a locally Specker algebra and $\phi(R)$ is semiprime, then $A$ is semiprime.  This is because the Pierce stalks of $A$ are from among those of $\phi(R)$, and they are all semiprime. In the same vein, since a homomorphic image of a pm-ring is a pm-ring, then if $\phi(R)$ is a pm-ring and $A\in \igr{R}$, then $A$ is a pm-ring; similarly when $\phi(R)$ is a clean ring.

One of the important module Pierce properties is the flatness of an $R$-module.   The next proposition will play a major role in later sections.

\begin{prop} \label{flP}   For a ring $R$ and an $R$-module $M$, the following are module Pierce properties: (1)~flatness of $M_R$, and (2)~faithful flatness of $M_R$.  \end{prop}

\begin{proof} (1)~Recall that $M_x= M/xM$. The criterion (e.g., \cite[Lem\-ma~19.19]{AF}) is used: If $\sum_{j=1}^na_jr_j=0$, $a_j\in M, r_j\in R$, there exist $v_1,\ldots, v_m\in M$ and $b_{ij}\in R$, $1\le i\le m$, $1\le j\le n$ with $\sum_{j=1}^nb_{ij}r_j=0$, $1\le i\le m$ and $\sum_{i=1}^m b_{ij}v_i=a_j$, $1\le j\le n$. 

Assume $M_R$ is flat. Consider a Pierce kernel $xM$, $x\in \spec \B(R)$; if $\sum_{j=1}^na_jr_j\in xM$ then there is $e \in \B(R)\setminus x$ where $e\sum_{j=1}^na_jr_j=0$ is an equation in $M$. The rest is easy. 

A typical Pierce sheaf argument is used here. If each $M_x$ is $R_x$-flat, suppose $\sum_{j=1}^na_jr_j=0$.  Then the criterion can be used with $R_x$ and $M_x$. Then, there is $e(x)\notin x$ which gives a set of equations in $M$.  However, the integer $m$, say $m(x)$, depends on $x$.  Then, compactness of $\spec \B(R)$ can be used to get a finite set of $e(x)$ which can be made orthogonal, say $\{e(x_1), \ldots, e(x_k)\}$. Put $m= \max \{m(x_1), \ldots, m(x_k)\}$. Then, all the elements $v_i$ and $b_{ij}$ can be found, using zeros, if necessary, where $m(x_i)<m$. 

(2)~It may be assumed that the module $M$ is flat. Then, the following criterion is used (e.g., \cite[Theorem~7.2(3)]{Ma}): a flat module $M$ is faithfully flat if and only if for every maximal ideal $\mf{m}$ of $R$, $\mf{m}M \ne M$. Assume that $M$ is not faithfully flat. Then for some maximal ideal $\mf{m}$, $\mf{m}M = M$. By Lemma~\ref{mp}(4), there is exactly one $x\in \spec \B(R)$ so that $\mf{m}\cap \B(R) = x$. It follows that $\mf{m}_xM_x= M_x$. 

In the other direction, if, for some $x\in \spec \B(R)$, $(M_x)_{R_x}$ is not faithfully flat, there is a maximal ideal $\mf{m}_x$ of $R_x$ so that $\mf{m}_xM_x = M_x$. Then $\mf{m}_x$ lifts to a maximal ideal $\mf{m}$ of $R$ with $xR\sbq \mf{m}$. For $a\in M$, there exists an equation  $a_x = \sum_{i=1}^k (r_ia_i)_x$, $r_i\in \mf{m}$ and $a_i\in M$ ($k$ depends on $a$). There exists $e\in \B(R)\setminus x$ such that $ae = \sum_{i=1}^k r_ia_ie\in \mf{m}M$. However, $a = ae+a(1-e)$ and $1-e\in x\sbq \mf{m}$.  Hence, $a\in \mf{m}M$.
\end{proof} 

The proposition will be applied to $R$-algebras, but first a lemma will be useful.  
 
\begin{lem} \label{intff}    Let $R\sbq S$ be rings such that $S$ is integral over $R$. If $S_R$ is flat then, $S_R$ is faithfully flat.\end{lem}

\begin{proof} This follows readily since $S$ has going-up with respect to $R$.  More precisely, if $\mf{m}$ is a maximal ideal of $R$, then there is $\mf{p}\in \spec S$ such that $\mf{p}\cap R= \mf{m}$. However, then $\mf{m}S \sbq \mf{p}S = \mf{p}$; by going-up, $\mf{p}$ is maximal; then the criterion of \cite[Proposition~7.1]{Ma} is used again.  
\end{proof} 

The following applies to locally Specker $R$-algebras. 

\begin{cor} \label{lSff}    Suppose that $A\in \igr{R}$.  If $A_{\phi(R)}$ is flat then, $A_{\phi(R)}$ is faithfully flat.  In particular (\cite[Proposition~7.1(1 and 2)]{BMO}), if $A$ is a locally Specker $R$-algebra, then $A_{\phi(R)}$ is faithfully flat. \end{cor}

\begin{proof} The first part is by Lemma~\ref{intff}.

By \cite[Theorem~3.3(2)]{BMO}, for each $x\in \spec \B(R)$, $A_x$ is a Specker $R_x$-algebra and, hence, $R_x$-free.  Then, by Proposition~\ref{flP}(2), $A_{\phi(R)}$ is faithfully flat (keeping in mind that a flat module $M_R$ such that $\bz \ne \ann_R M$ is an idempotent ideal cannot be faithfully flat module, by \cite[Theorem~7.2(2)]{Ma}). 
 \end{proof} 

Note  that \cite[Proposition~7.1(2)]{BMO} should be stated: \emph{if $A$ is a locally Specker $R$-algebra, $A_R$ is flat and $A_{\phi(R)}$ is faithfully flat.}  This is because $\ann_RA$ is generated by idempotents. 

While on the subject of locally Specker algebras, the following should be noted.  If $A$ is a locally Specker $R$-algebra, the Pierce stalks of $A$ are from among those of $R$, in fact of $\phi(R)$ (\cite[Theorem~3.3]{BMO}); hence, the following.

\begin{prop} \label{lSp}  Let $\mc{P}$ be a ring Pierce property. If $A$ is a locally Specker $R$-algebra, then $A$ has property $\mc{P}$ if and only if $\phi(R)$ has property $\mc{P}$.  \end{prop}

An example of this is contained in the following. Part (iii) of \ref{known} generalizes \cite[Theorem~4.3(1)]{BMMO} to awB rings (every wB ring is awB but not conversely e.g., \cite[Example~5.4]{Be}), and, moreover, to locally Specker algebras. 

\begin{cor} \label{awBSp}    A locally Specker $R$-algebra $A$ is awB if and only if $\phi(R)$ is awB. \end{cor} 

\begin{proof}  If all the Pierce stalks of $A$ are domains, so are those of $\phi(R)$; and conversely if all the Pierce stalks of $\phi(R)$ are domains, so are those of $A$.  
\end{proof}

More ring Pierce properties are now examined. The following sorts of rings have intrinsic interest and some are studied further in later sections. 

The question of whether or not $\minspec R$ (with the Zariski topology) is compact is a frequent topic, especially in  connection with the regularity of $\qcl (R)$.  

\begin{prop} \label{comp}  For a ring $R$, if $\qcl(R)$ is regular then $\minspec R$ is compact. \end{prop}

\begin{proof} Here $R$ is semiprime.  The result \cite[Theorem~2.2]{C} says that if any regular $V$ between $R$ and $Q(R)$ is $R$-flat then any such regular ring is $R$-flat.  This applies when $\qcl(R)$, which is always $R$-flat, is regular. Hence, the regular ring $Q(R)$ is $R$-flat.   Then by \cite[Theorem~3.1]{M}, $\minspec R$ is compact.  
\end{proof} 

The converse of Proposition~\ref{comp} is ``almost'' true (to quote \cite{HJ}), but the following is needed: the \emph{annihilator condition} (a.c.).  A semi\-prime ring $R$ is said to satisfy the \emph{annihilator condition} (a.c.) if for $r,s\in R$ there is $t\in R$ with $\ann_R r\cap \ann_R s= \ann_R t$ (\cite[Definition~3.2]{HJ}). 

The following could be said to be implicit in \cite{HJ}, but no proof is given there.    

\begin{prop} \label{qclreg}   For a ring $R$, $\qcl(R)$ is regular if and only if $\minspec R$ is compact and $R$ satisfies the a.c. \end{prop}

\begin{proof} It can be assumed that $R$ is semiprime since the definition of the a.c.\ requires it.  In \cite[Theorem~3.4]{HJ} ((c) $\Leftrightarrow$ (a)), the two conditions in the statement are shown to be equivalent to: for every $r\in R$, there is $r' \in R$ such that $\ann_R(\ann_R r') = \ann_R r$.  

Assume first that $\qcl(R)$ is regular. Let $r\in R$. There is an idempotent $e = a/b\in \B(\qcl(R))$ which generates $r\qcl(R)$. Put $r' = b-a$, and $\ann_{\qcl(R)} r = (1-e)\qcl(R)$.  Moreover, $\ann_{\qcl(R)} (b-a) = e\qcl(R)$. Hence, $\ann_{\qcl(R)}(\ann_{\qcl(R)} r') =\ann_{\qcl(R)}r$. Then, taking intersections with $R$, the result follows.

In the other direction, take $r, r'\in R$ such that $\ann_R(\ann_R r')$ $ = \ann_R r$. It follows that $rr' = 0$.  Then, $r+r'$ is a non-zero divisor in $R$.  Indeed, let $t\in R$ be such that $t(r+r') = 0$; since $R$ is semiprime, $tr = tr' = 0$. This means that $t\in \ann_Rr'$ and $t\in \ann_Rr = \ann_R(\ann_Rr')$. Hence, $t^2=0$ and $t=0$. 

Then, there is $q\in \qcl(R)$ with $(r+r')q=1$. It follows that $r = r(r+r')q = r^2q$. In addition a typical element of $\qcl(R)$ is of the form $r/d$, $d\in R$ a non-zero divisor.  Hence, $r/d = (r^2/d^2) qd$, showing that $\qcl(R)$ is regular. 
\end{proof}

\begin{rem} \label{Pcom}    If $R$ has $\minspec R$ compact, then so do all its Pierce stalks.  The converse is false.  \end{rem}

\begin{proof} Always, $\minspec R$ is Hausdorff (\cite[Corollary~2.4]{HJ}). If $\minspec R$ is compact then  if $x\in \spec \B(R)$, by Lem\-ma~\ref{mp}(1), the closed Zariski set $\{\mf{p}\in \minspec R\mid \mf{p}\spq xR\}$ is homeomorphic to  $\minspec R_x$. This is a closed set in a compact Hausdorff space and, hence, is compact.

On the other hand, the example \cite[Example~3.2]{NR}, $R= \C(\beta\N\setminus \N)$ ($\beta$ is the \v{C}ech-Stone compactification) is an awB ring making each $\minspec R_x$ compact (one element). However, this ring satisfies the a.c.\ and if $\minspec R$ were compact, $\qcl(R)$ would be regular.  However, $\qcl (R) =R$ and it is not regular. 
\end{proof}

Remark~\ref{Pcom} has a parallel result for the a.c.  

\begin{thm} \label{acP}  Let $R$ be a semiprime ring.  Then, (1)~if $R$ satisfies the a.c.\ so do all of its Pierce stalks, (2)~the converse of (1) is false, even for awB rings.  \end{thm}

\begin{proof} (1)~It must be shown that in a Pierce stalk $R_x$, for $r_x, s_x\in R_x$, there exists $t_x\in R_x$ such that $\ann_{R_x}t_x = \ann_{R_x}r_x\cap \ann_{R_x}s_x$. Since $R_x$ is a classical localization (Lemma~\ref{mp}(2)), annihilators of elements commute with the localization (e.g., \cite[Proposition~3.14]{AtM}); then lift $r_x$ and $s_x$ to $r, s\in R$. Then there is $t\in R$ with $\ann_Rr\cap \ann_Rs = \ann_Rt$.  From this $\ann_{R_x}r_x\cap \ann_{R_x}s_x = \ann_{R_x}t_x$.  Thus, $R_x$ satisfies the a.c.

(2) To show this, a digression is necessary.  In \cite[Example~3.3]{HJ} there is an example, due to H.\ Flanders, of a semiprime ring that does not have the a.c.  It can be used to show (2), but another example will be used here.  In \cite[Proposition~2.3]{Lu} a different construction is used to find examples.  It works as follows: Let $K$ be an algebraically closed field and $D = K[X,Y]$. Put $\mc{M}$ to be the set of maximal ideals of $D$, and $I = |\mc{M}| \times \N$,  where $i\in I$ has the form $(\alpha, n), n\in \N$ and $M_\alpha\in \mc{M}$. The ring $R$ will be a subring of $\Pi= \prod_{i\in I} D/M_\alpha$, when $i=(\alpha, n)$. Put $B= \bigoplus_{i\in I}D/M_\alpha$. The subring $A\subset \Pi$ is formed from $\phi\colon D\to \Pi$ where $\phi(d)$ is the element whose $i$-component, $i=(\alpha, n)$ is $d+M_\alpha$. Finally $R= A+B$.  It is shown in \cite[Proposition~2.3]{Lu} that $R$ does not have a.c. 

It will be shown here that the base space of the Pierce sheaf of $R$ is the one-point compactification, $I\cup \{\infty\}$, of the discrete space $I$, with Pierce stalks all fields except for $R_\infty \cong D$. In other words all the Pierce stalks have the a.c.   In $\B(R)$, for each $i\in I$ there is an atom $e_i$ which is 0 except in the $i$-component where it is 1. If $i=(\alpha,n)$,  $x_i=(1-e_i)\B(R)\in \spec \B(R)$ and $R_{x_i}=D/M_\alpha$. 

In \cite[Proposition~2]{BS}, it is shown exactly when a ring has a Pierce sheaf whose base space is of the form $I\cup \{\infty\}$, the one-point compactification of the discrete space $I$ (already seen in simplified form in Construction~\ref{con2}(2)). What is needed is to have a ring homomorphism $\psi \colon D\to (\prod_{i\in I} D/M_\alpha)/ (\bigoplus_{i\in I}D/M_\alpha)$, along with a condition on the idempotents of $D$.  Such a homomorphism is already available, namely, $\phi$ followed by dividing out by the direct sum; since $D$ is a domain, the condition about idempotents is satisfied.  It remains to show that $\psi$ is a monomorphism, making $R_\infty \cong D$. 

For each pair $(a,b)\in K\times K$, $M_{a,b}= \{d\in D\mid d(a,b) =0\}$, is a maximal ideal in $D$; in the set of maximal ideals of $D$, label this one $\alpha(a,b)$. Put $L = \{(\alpha(a,b),n)$, $(a,b)\in K\times K$  and $n\in \N\}\subset I$.  

Write $d=d(X,Y) = \sum_{m,n\ge 0}a_{m,n}X^mY^n$, $a_{m,n} \in K$ and suppose that $\psi(d) =0$. Let $\N_0$ be $\N\cup\{0\}$. 

Since $d$ is zero at all but finitely many of the points of $I$,  in particular, $d$ is zero at all but finitely many points in $L$. Fix $0\ne a\in K$. Then, for infinitely many $b\in K$, $d(a,Y)$ has zero value.  Hence, $d(a,Y)$ is the zero polynomial.  Then, for each $n\in \N_0$, $\sum_{m\in \N_0}a_{m,n}a^m=0$. However, this is true for infinitely many $a$, making, for each $n\in \N_0$, $a_{m,n}=0$. In other words, $d=0$. 

All the Pierce stalks of $R$ are fields except one which is the domain $D$. All have the a.c., but $R$ does not. 
\end{proof}

The ring $R$ of Theorem~\ref{acP}(2) has other properties worth mentioning that are related to some of the topics above.  

\begin{cor} \label{qcl=R}   The ring of Theorem~\ref{acP}(2) has the property that $\qcl(R) =R$; it is not regular. Moreover, $R$ is awB but not wB. In addition, $\minspec R$ is not compact.\end{cor}

\begin{proof} All the Pierce stalks of $R$ are domains, making $R$ awB.  

Consider an element of $r\in R$ that is a non-zero divisor.  There is some $d\in D$ such that, except at finitely many places, $r_i = d+ M_\alpha$, where $i=(\alpha, n)$. If, for some $\alpha$, $d\in M_\alpha$, $r$ would be zero at infinitely many $i\in I$, and, hence, a zero divisor. Thus, $d$ is not in any maximal ideal and is a unit in $D$. This makes $r_\infty = d$ a unit.   But then, $r$ can be seen to have an inverse in $R$ since $d$ does.  Hence, $\qcl(R)$ is not regular.

The next statement follows because in a wB ring, $\qcl$ is regular. Finally, awB plus $\minspec R$ compact implies $R$ is a wB ring.   This is by the remarks, due to \cite{AM}, before \cite[Proposition~1.4]{NR}.  Note that this requires a condition called WDML, but that is implied by awB (\cite[ Proposition~1.4]{NR}). 
\end{proof}

  Pure ideals are characterized in many ways (see e.g., \cite[1 \S11]{Ste} and also before Lemma~\ref{mp} above), but the basic one is that an ideal $I$ of $R$ is pure if, and only if, $(R/I)_R$ is flat. In \cite[Section~7]{BMO}, rings whose pure ideals are generated by idempotents play an important role in $\igr{R}$. Those rings have various names but, here, the terminology of \cite{V} is used.

\begin{defn}\cite[\S~3]{V} \label{fR}   A ring $R$ in which every pure ideal is generated by idempotents is called an \emph{f-ring} \end{defn}

It will be shown here that being an f-ring is a ring Pierce  property.  This will expand the class of known f-rings (see a partial list in \cite[Remark~7.10]{BMO}, and the references quoted there).  Note that an indecomposable ring $R$ is an f-ring exactly when $R$ does not have any non-trivial pure ideals.

\begin{thm} \label{fP'}  The property of being an f-ring is a ring Pierce property. \end{thm}

\begin{proof}   First suppose that $R$ is an f-ring. Let $x\in \spec \B(R)$ and $I_x$ be a pure ideal of $R_x$.  Consider $\phi_x\colon R\to R_x \to R_x/I_x$ and $K =\ker \phi_x$.  Then $K$ is a pure ideal of $R$ containing $xR$.  This is because $R_x/I_x$ is a flat $R_x$-module and $R_x$ is a flat $R$-module.  Hence, $K$ is generated by idempotents, say $K=ER$, $E=\B(R)\cap K$.  Note that the image of $K$ in $R_x$ is $I_x$, and that $xR\sbq K$. 

If $xR=K$, then $I_x=\bz$. Otherwise, there is $e\in E\setminus x$. In that case, the image of $K$ in $R_x$ is all of $R_x$.  Hence, $I_x=\bz$ or $I_x=R_x$. 

For the converse, where each $R_x$ is an indecomposable f-ring, consider a pure ideal $I$ of $R$ and the pure ideal $ER$, where $E= I\cap \B(R)$. It will be shown that $I=ER$. Take $x\in \spec \B(R)$.

The completion of the proof follows a schema suggested by the referee and is more efficient than in the original. 

The fact that for ideals $K$ and $L$ of $R$, if for all $x\in \spec \B(R)$, $K_x=L_x$ then $K=L$ is used.  

For all $x\in \spec \B(R)$, $I_x=R_x$ or $I_x = \bz_x$.  In case $I_x=\bz_x$, it follows that $(ER)_x= \bz_x$. If $I_x=R_x$, there is $a\in I$ such that $(a-1)_x=0_x$. Then, there is $e\in \B(R)\setminus x$ such that $ae=e$.  It follows that $e\in E$ and, hence, $(ER)_x=R_x$. Hence, $I = ER$.  \end{proof}

Some consequences of this theorem will be elaborated upon in the next section.

Although these properties are not used in what follows, it can also be shown that the following are also ring Pierce properties: (1)~arithmetical, (2)~WDML of \cite{NR}, (3)~B\'{e}zout, and   (4)~seminormal (see \cite{Sw}). \epar
 
\setcounter{section}{3} \setcounter{thm}{0}
\noindent\tbf{3. More on f-rings.} This section is devoted to some corollaries of Theorem~\ref{fP'} and to the connection between f-rings and flat extensions by idempotents.  

Recall that if $R$ is a ring such that each pure ideal is generated by a single idempotent, $R$ is called an \emph{F-ring} (see \cite{V}).  

\begin{cor} \label{Fring}   If $R$ is an F-ring, then $R$ is a direct product of a finite number of indecomposable rings without non-trivial pure ideals; and conversely.\end{cor}

\begin{proof} By \cite[Theorem~7.5 and Remark~7.12]{BMO}, when $R$ is an F-ring, $\B(R)$ is finite.  Then, Theorem~\ref{fP'} gives the result.  

The converse is obvious since a finite product of F-rings is an F-ring. 
\end{proof} 

By Corollary~\ref{Fring}, the property of being an F-ring is clearly not a ring Pierce  property.

\begin{cor} \label{noethf}   Let $R$ be a ring whose Pierce stalks are noetherian.  Then, $R$ is an f-ring, and for any $A\in \igr{R}$, $A$ is an f-ring. \end{cor}

\begin{proof} By \cite[Proposition~4.5 and Corollary~4.7]{Jp} all the Pierce stalks of $R$ are f-rings. Then Theorem~\ref{fP'} gives the result about $R$. However, by Lemma~\ref{addi}, this carries over to the Pierce stalks of $A$.  
\end{proof} 

Of course, Corollary~\ref{noethf} holds if $R$ is itself noetherian, but, here, $R$ is an F-ring. There are non-noetherian rings whose Pierce stalks are noetherian, e.g., any regular ring that is not a finite product of fields; or the ring of Theorem~\ref{acP}(2). See also Example~\ref{Prf1}.

It was known that a wB ring is an f-ring.  However, the following is more general.

\begin{cor}\label{awBf}  If $R$ is an awB ring, then it is an f-ring. \end{cor}

\begin{proof}  The awB rings are exactly those whose Pierce stalks are domains (\cite[Theorem~2.2]{NR}).  Hence, because domains are f-rings, Theorem~\ref{fP'} applies.
\end{proof} 

The following is known but also follows readily from Theorem~\ref{fP'}.

\begin{cor}\label{clf}   If $R$ is a clean ring then $R$ is an f-ring. \end{cor} 

\begin{proof} It is readily seen that a local ring is an f-ring.  However, clean rings are exactly those whose Pierce stalks are local. 
\end{proof}

The following is a scheme for constructing rings where the Pierce stalks are f-rings.  One of the many characterizations of a pure ideal $I$ in $R$, is that if $a\in I$ then there is $b\in I$ with $ab = a$ (already used in Lemma~\ref{mp}(1)). This gives rise to the notion of a \emph{pure sequence} (see \cite[\S4]{Jp}). This is a sequence $a_1, a_2, ...$ in a ring $R$ such that $\langle a_i\rangle\, \subset\, \langle a_{i+1}\rangle $ and, for all $i$, $a_{i+1}a_i = a_i$.   If $R$ is such that there is $n\in \N$  such that every pure sequence terminates in at most $n$ steps (meaning $a_{n+1}=a_{n+2} =\cdots$ and $a_{n+1}\in \B(R)$), then $R$ is said to have \emph{finite pure length} (\cite[\S4]{Jp}).  If $R$ is a ring of finite pure length then $R$ is an f-ring (but this is not a necessary condition, see, e.g., Example~\ref{Prf1}). Every noetherian ring is of finite pure length (\cite[Proposition~4.5]{Jp}). 

There are f-rings such that all pure sequences are finite but that there is no bound on the lengths, using Construction~\ref{con2}(2).  There are also f-rings where there are infinite pure sequences, using \cite[Proposition~4.6]{Jp}.

\begin{ex} \label{Prf1}   Let $\{R_\alpha\}_{\alpha \in A}$, $A$ some set, be a family of indecomposable f-rings that have a common subring $R$ that has finite pure length.  Let $S$ be a subring of $\prod_{\alpha\in A} R_\alpha$ of those elements whose components are in $R$ except for finitely many.  Then, $S$ is an f-ring. \end{ex} 

\begin{proof} This is an example of Construction~\ref{con2}.  It suffices to show that every Pierce stalk of $S$ is an f-ring.  The space $\spec\B(S)$ is $\beta A$. The Pierce stalks corresponding to the elements of $A$ are the $R_\alpha$. The other Pierce stalks are ultrapowers of $R$.  It suffices to show that they are f-rings. However, an arbitrary power of copies of $R$ is an f-ring (\cite[Proposition~4.6]{Jp}) and, thus, ultrapowers of $R$ are f-rings by Theorem~\ref{fP'} and Construction~\ref{con1}. \end{proof}

One case of Example~\ref{Prf1} would be if one of the  $R_\alpha$ were noetherian. Then, any common subring would be of finite pure length (\cite[Proposition~4.5]{Jp}). 

In a product $P$ of indecomposable rings $\{R_\alpha\}_{\alpha \in A}$, the Pierce stalks of $P$ are the rings $R_\alpha$ and ultraproducts of them, as in Construction~\ref{con1}. If the $R_\alpha$ are f-rings for a reason that is preserved by ultraproducts, then $P$ is also an f-ring.  This would apply if all the $R_\alpha$ were domains; although that is already covered by Corollary~\ref{awBf}. 

However, a product of f-rings is not necessarily an f-ring.  Indeed, even a product of noetherian rings need not be an f-ring -- see \cite[Remark~\S4]{Jp} for an example. 

The more general question of which products of  (indecomposable) f-rings are f-rings remains unresolved, as mentioned above.  However, one case has a positive answer: a product of awB rings is readily seen to be an awB ring; these are f-rings.

In \cite[\S~7]{BMO} there is a discussion of flat extensions by idempotents and their relationship to f-rings.  Then, \cite[Theorem~7.9]{BMO} says that the following are equivalent: (a)~$R$ is an f-ring, and (b)~an $R$-algebra $A$ is locally Specker if and only if $A\in \igr{R}$ and $A_R$ is flat.  This theorem will be reproved here using Theorem~\ref{fP'}.  Before then, there is a presentation of rings that are not f-rings.

There is a reference to a family of examples in \cite[$\dagger$~2.9]{J} of indecomposable rings with non-trivial pure (there called ``neat'') ideals clearly not idempotent generated.  These are recalled here.  Let $X$ be a connected completely regular space and $\cx$ the ring of continuous real valued functions on $X$ (see \cite[Chapter~3]{GJ}). Fix $\alpha \in X$ and let $\mc{O}_\alpha$ be the ideal of functions zero on a neighbourhood of $\alpha$. It is shown in \cite{J} that $\mc{O}_\alpha$ is a pure ideal not generated by idempotents. This line of reasoning produces examples of flat extensions by idempotents that are not (locally) Specker; see the next lemma.

\begin{lem} \label{indnf}   Let $R$ be an indecomposable ring that is not an f-ring. Then, there is $A\in \igr{R}$ such that $A_R$ is flat but $A$ is not a Specker $R$-algebra.  \end{lem}

\begin{proof} There is a non-trivial pure ideal $I$ in $R$.  Consider $T= R/I$. Since $T_R$ is flat, all the Pierce stalks of $T$ are $R$-flat.  Let $T_y$ be one of them.  Then, $T_y$ is a flat indecomposable $R$-algebra. Moreover, $T_y$ is a homomorphic image of $R$ (since it is a homomorphic image of $R/I$).  Then, the kernel, $I(y)$, of the natural homomorphism $R\to T_y$ is a pure ideal containing $I$.  Now, consider $A=R\times R/I(y)$, a flat $R$-algebra. It follows that $\B(A) = \{(0,\bar{0}), (0,\bar{1}), (1, \bar{0}), (1, \bar{1})\}$, and the Pierce stalks of $A$ are $R$ and $R/I(y)$. It follows that  $A$ is generated by $R$ and $(0,\bar{1})$; i.e., $A\in \igr{R}$.   If $A$ were a Specker $R$-algebra, $A(0,\bar{1}) \cong R/I(y)$ would be isomorphic to $R$, as $R$-modules (\cite[Lemma~2.6]{BMMO}).  This is clearly not the case.  Hence, $A$ is the example.
\end{proof}

One source of rings $R$ as in Lemma~\ref{indnf} is $\cx$ where $X$  is a connected metric space, e.g., $X=\R$. 

A proof of Lemma~\ref{indnf} without the restriction to $R$ being indecomposable is possible.  However, that follows directly from the following proposition: that is, if $R$ is not an f-ring there is an $A\in \igr{R}$ such that $A_R$ is flat but $A$ is not a locally Specker $R$-algebra.

Lemma~\ref{indnf}, along with Theorem~\ref{fP'}, allows an alternate proof to a known result, \cite[Theorem~7.9]{BMO}, not using the powerful \cite[Lemma~7.7]{BMO}.  

\begin{prop}\cite[Theorem~7.9]{BMO} \label{BMO7.9}   The following statements are equivalent for a ring $R$. 

(a) $R$ is an f-ring.

(b) A ring $A$ is a locally Specker $R$-algebra if and only if $A_R$ is flat and $A\in \igr{R}$. \end{prop}

\begin{proof} (a) $\Rightarrow$ (b) Assume first that $R$ is an f-ring.  Let $A = R[\B(A)]$ and assume that $A_R$ is flat. Take $y\in \spec \B(A)$ and $x= y\cap \B(R)$. Then, there is a surjection $R_x\to A_y$ (Lemma~\ref{addi}). By Proposition~\ref{flP}(1), $A_y$ is $R$-flat,  and then it is $R_x$-flat, as well.  But, by Theorem~\ref{fP'}, $R_x$ has no non-trivial pure ideals, and thus the kernel of $R_x\to A_y$ is zero.  This makes $A$ a locally Specker $R$-algebra (\cite[Theorem~3.3(3)]{BMO}). 

If $A$ is a locally Specker $R$-algebra, it is in $\igr{R}$ and $A_R$ is flat.  

(b) $\Rightarrow$ (a) Assume that (b) holds. By Theorem~\ref{fP'}, it suffices to show that for $x\in \spec \B(R)$,  $R_x$ has no non-trivial pure ideals.  If $\bz\ne I\subset R_x$ is a pure ideal, then Lemma~\ref{indnf} gives a ring $A\in \igr{R_x}$ that is not a Specker $R_x$-algebra, but $A_{R_x}$ is flat. This contradicts (b), showing that $R_x$ is an indecomposable f-ring. Hence, $R$ is an f-ring.
\end{proof}

There are more applications of Theorem~\ref{fP'}.

Both the properties pm-ring (Proposition~\ref{known}(iv)) and f-ring (Theorem~\ref{fP'}) are ring Pierce properties.  It will turn out that another characterization of a clean ring is that it is both a pm-ring and an f-ring.

\begin{prop} \label{pmfind}   Let $R$ be an indecomposable ring. It is both a pm-ring and an f-ring if and only if it is local. \end{prop}

\begin{proof} A local ring is both a pm-ring and an f-ring. 

For the converse, first let $S$ be any pm-ring. Then, there is a continuous surjection $\phi \colon \spec S \to \maxspec S$.  If $C$ is a clopen set in $\maxspec S$, $\phi\inv (C)$ is a clopen set in $\spec S$ and, hence, is $D(e)$ for some $e\in \B(S)$. If $C'=\maxspec S\setminus C$, then $\phi\inv(C')= D(f)$, some $f\in \B(S)$. Since $\phi\inv(C)\cap \phi\inv(C')$ are disjoint, $ef=0$. 

Now return to $R$, as in the statement.  According to \cite[Corollary, \S2.4]{DeM}, $\maxspec R$ is totally disconnected. If $\mf{m}$ and $\mf{m}'$ are distinct maximal ideals of $R$, there is a clopen set $C\subset \maxspec R$ with $\mf{m}\in C$ and $\mf{m}'\notin C$. There are orthogonal idempotents $e,f\in \B(R)$ such that $C=D(e)\cap \maxspec R$ and its complement is $D(f)\cap \maxspec R$.  Since $\B(R) =\{0,1\}$, this leads to a contradiction since 0 is in every maximal ideal.  Hence, there is only one maximal ideal.  
\end{proof}

\begin{cor}\label{pmf=c}   A ring $R$ is a pm-ring and an f-ring if and only if $R$ is clean. \end{cor}

\begin{proof} A clean ring is known to be both a pm-ring and an f-ring.

Conversely, since both properties are ring Pierce properties, it is sufficient to look at a Pierce stalk, $R_x$. By Proposition~\ref{pmfind}, $R_x$ is a local ring. Hence, $R$ is a clean ring. 
\end{proof}

Note that any product of clean rings is a clean ring.

As mentioned above, if $R$ is a clean ring then $\maxspec R$ is totally disconnected. It would be interesting to know if there exists an f-ring $R$ with $\maxspec R$ totally disconnected and $\text{T}_2$, but $R$ is not a clean ring.\epar

\setcounter{section}{4} \setcounter{thm}{0}
\noindent\tbf{4. A weak form of going-down vs flatness.}  This section has material about the going-down property of an extension by idempotents. Of course, if the extension is flat, going-down holds (e.g., \cite[Theorem~9.5]{Ma}), but not the converse. It is this converse that is discussed. 

In particular, if $A$ is a locally Specker $R$-algebra, then $A_R$ is flat by \cite[Proposition~7.1(2)]{BMO}.  In addition, since a locally Specker $R$-algebra $A$ is generated by idempotents and is $R$-flat, $A$ has both going-up and going-down with respect to $R$.  

In this context there will be some weakening of ``going-down'' that will be considered, and will often be all that is needed to get flatness.  

\begin{defn} \label{wgd}    Let $A$ be a ring extending a ring $R$.  Then, $A$ satisfies \emph{weak going-down} with respect to $R$ if $\mf{q}\in \spec A$ and $\mf{q}\cap R = \mf{p}$ and $\mf{p}'$ is a minimal prime lying in $\mf{p}$ then there is a minimal prime $\mf{q}'\sbq \mf{q}$ with $\mf{q}'\cap R=\mf{p}'$.   \end{defn}

Note that if $R\sbq A$ are domains, then $A$ always satisfies weak going-down with respect to $R$.  However, there are known integral extensions of domains that do not satisfy full going-down, e.g., \cite[Chapter~2, Example~5.1]{Ma2}. Hence, the concept of Definition~\ref{wgd} is strictly more general than going-down.

\begin{lem} \label{l-wgd}    Let $R\sbq A$ be rings so that $A$ has weak going-down with respect to $R$. (1)~Minimal primes of $A$ contract to minimal primes of $R$. (2)~If, in addition, $A$ is integral over $R$ and $R\sbq T \sbq A$ are rings, then, $T$ has weak going-down with respect to $R$. \end{lem}

\begin{proof} (1)~Let $\mf{q}\in \minspec A$ and $\mf{p}=\mf{q}\cap R$. There is $\mf{p}'\in \minspec R$ with $\mf{p}'\sbq \mf{p}$. By weak going-down, there is $\mf{q}'\sbq \mf{q}$ with $\mf{q}'\cap R=\mf{p}'$. However, $\mf{q}$ is minimal and so $\mf{q}'=\mf{q}$, and, hence, $\mf{p}'=\mf{p}$. 

(2)~Let $\mf{q}\in \spec T$ with $\mf{p}=\mf{q}\cap R$. By integrality, there is $\mf{q}'\in \spec A$ with $\mf{q}'\cap T=\mf{q}$. Let $\mf{p}'\sbq \mf{p}$ be a minimal prime. By weak going-down, there is $\mf{r}\in \minspec A$ with $\mf{r}\sbq \mf{q}'$ and $\mf{r}\cap R=\mf{p}'$. Then, $\mf{r}\cap T=\mf{s}\sbq \mf{q}$. Let $\mf{s}'\in \minspec T$ with $\mf{s}'\sbq \mf{s}$. It follows that $\mf{s}'\cap R= \mf{p}'$, and $T$ has weak going-down with respect to $R$.
\end{proof} 

It will be shown that in some situations in extensions by idempotents, that this weaker form of going-down implies flat (and, hence, full going-down).  

A special case is needed where weak going-down implies flatness (and even more).  This is a key to a more general result. 

\begin{lem} \label{fl}   Suppose $R$ is a semiprime local ring and $R\sbq A=R[E]\in \igr{R}$,  $E\sbq \B(A)$, so that $A$ satisfies weak going-down with respect to $R$.  Then, every Pierce stalk of $A$ is a copy of $R$, and $A$ is semiprime. Moreover, $A=R[\B(A)]_{Sp}$ and $A_R$ is free.\end{lem} 

\begin{proof} Denote the unique maximal ideal of $R$ by $\mf{m}$.  Since $R$ is semiprime, there are minimal primes $\mf{p}_i, i\in I$ with $\bigcap_{i\in I}\mf{p}_i = \bz$.  Now consider a Pierce kernel $yA$ of $A$ and choose a maximal ideal $\mf{n}$ with $yA\sbq \mf{n}$. It follows that $\mf{n}\cap R= \mf{m}$ (by going-up). 

Now each minimal prime of $R$, $\mf{p}_i\sbq \mf{m}$ and so, by weak going-down, there is a minimal prime $\mf{q}_i\sbq \mf{n}$ in $A$ lying over $\mf{p}_i$.  Also, $yA\sbq \mf{q}_i$ by Lemma~\ref{mp}(1).

Since $yA\sbq \bigcap_{i\in I}\mf{q}_i$, $yA\cap R \sbq \bigcap_{i\in I}\mf{q}_i \cap R = \bigcap_{i\in I} \mf{p}_i =\bz$, showing that $yA\cap R=\bz$. Then, $R\to A \to A/yA$ is onto with kernel $yA\cap R$, showing that $A/yA\cong R$, in particular as $R$-modules.

Since $R$ is indecomposable, locally Specker extensions of $R$ are, in fact, Specker. The result follows from \cite[Theorem~3.3(3)]{BMO}, which says that $A$ is  Specker, hence $A_R$ is free. \end{proof}

Before the next two propositions about weak going-down, a lemma will be useful.

\begin{lem} \label{loclem} Let $R\sbq T$ be an extension of rings and $0\notin S\subset R$ be multiplicatively closed. If $T$ satisfies weak going-down with respect to $R$ then so does $S\inv T$ with respect to $S\inv R$. \end{lem} 

\begin{proof} The prime ideals of $S\inv T$ correspond to the prime ideals of $T$ not meeting $S$, modulo the kernel of $T\to S\inv T$.  That is, if $\mf{a}\in \spec T$ and $\mf{a}\cap S=\vide$, then $\mf{a}$ corresponds to $S\inv \mf{a}$. This correspondence preserves and reflects inclusions. Similarly for $S\inv R$. If $\mf{b}\in \minspec S\inv T$, then, via the correspondence, it comes from a minimal prime in $T$. 

Now let $\mf{q}\in\spec S\inv T$ and $\mf{p}=\mf{q}\cap S\inv R$. These primes correspond to primes $\mf{p}'$ of $R$ and $\mf{q}'$ of $T$ according to the arrangement above. Now take a minimal prime $\mf{p}_1$ of $S\inv R$, with $\mf{p}_1\sbq \mf{p}$.  Then, $\mf{p}_1$ corresponds to a minimal prime $\mf{p}_1'$ in $R$, with $\mf{p}_1'\sbq \mf{p}'$.  By weak going-down in $T$ with respect to $R$, there is $\mf{q}_1'\in \minspec T$ with $\mf{q}_1'\sbq \mf{q}'$ and $\mf{q}_1'\cap R = \mf{p}_1'$.  Then, $\mf{q}_1'$ corresponds to $\mf{q}_1\in \minspec S\inv T$ and $\mf{q}_1\cap S\inv R = \mf{p}_1$, showing weak going-down of $S\inv T$ with respect to $S\inv R$. \end{proof}

\begin{prop} \label{spclean}  Suppose $R$ is a semiprime clean ring and $R\sbq A=R[E]\in \igr{R}$, $E\sbq \B(A)$, so that $A$ satisfies weak going-down with respect to $R$.  Then, every Pierce stalk of $A$ is a copy of a Pierce stalk of $R$. Moreover, $A$ is a locally Specker $R$-algebra and $A_R$ is flat.     \end{prop}

\begin{proof}  It is first needed to reduce the problem to an algebra over a semiprime local ring.  Choose $x\in \spec\B(R)$ and consider the extension by idempotents $A_x=A/xA$ of $R_x$. Note that $xA\cap R=xR$ and so $R_x \sbq A_x\in \igr{R_x}$.

Since $R_x$ and $A_x$ are classical localizations of $R$ and $A$, respectively, if $A$ satisfies weak going-down with respect to $R$, then, by Lemma~\ref{loclem}, $A_x$ satisfies weak going-down with respect to $R_x$. 

It follows, by Lemma~\ref{fl}, that $A_x$ is a Specker algebra over $R_x$. Hence, by \cite[Theorem~3.3(2)]{BMO}, $A$ is a locally Specker $R$-algebra, making, in particular, $A_R$ flat.
\end{proof}

For more general semiprime rings, the technique will be to reduce to the case of Lemma~\ref{fl}.  

 \begin{thm} \label{gdP}   Let $R$ be a semiprime ring and $R\sbq A=R[E]\in \igr{R}$, $E\sbq \B(A)$. If $A$ has weak going-down with respect to $R$ then $A_R$ is flat. \end{thm}
 
 \begin{proof} The technique is to reduce to the local case and apply Lemma~\ref{fl}. Choose $\mf{p}\in \spec R$ and look at the local ring $R_\mf{p}$ and the module (ring) of fractions $A_\mf{p}$. 
 
The following will be needed. 

 If $R\sbq S$ is an extension ring and $S$ satisfies weak going-down with respect to $R$ and $\mf{p}\in \spec R$, then $S_\mf{p}$ satisfies weak going-down with respect to $R_\mf{p}$. This is by Lemma~\ref{loclem}. 
   
 From the above, $A_\mf{p}$ has weak going-down with respect to $R_\mf{p}$. Lem\-ma~\ref{fl} now shows that $A_\mf{p}$ is flat over $R_\mf{p}$ (in fact, free). However, flatness of a module is a local property (see, e.g., \cite[page~26 and Theorem~7.1]{Ma}) showing that $A_R$ is flat. 
 \end{proof}
  
 It can be shown that the hypothesis ``semiprime'' in Theorem~4.5 is necessary.  
  
 The following is implicit  in \cite{BMO}, but is worth stating  explicitly. 

 \begin{prop} \label{finE}  Suppose $R$ is a ring and $A\in \igr{R}$.  Then, the following are equivalent:
 
 (i) $A$ is flat over $R$, 
 
 (ii) For every $F\sbq \B(A)$,  $R[F]$ is flat over $R$,

(iii) For each $e\in \B(A)$, $R[e]$ is flat over $R$. 
\end{prop} 
 
 \begin{proof} By \cite[Lemma~7.7]{BMO}, each of these three statements is equivalent to: for all $e\in \B(A)$, $\ann_Re$ is a pure ideal in $R$. \end{proof}
 
 \setcounter{section}{5} \setcounter{thm}{0} 
 \noindent\tbf{5. Adding idempotents and semiprime.} Semiprime commutative rings are an important subclass of $\mc{CR}$.  In this brief section, it will be first shown that adding idempotents to a semiprime ring does not always produce a semiprime ring.  Then, there are some cases where such extensions do produce semiprime rings.  There is more about semiprime rings in Section~6.
 
Clearly the property of being semiprime is both a Pierce property (Proposition~\ref{known}(v)) and a local property (meaning $R$ is semiprime if and only if for each $\mf{p}\in \spec R$, $R_\mf{p}$ is semiprime).  
  
 \begin{ex} \label{nsp}   There is an example of a semiprime ring $R$ and an extension $A=R[e]$, where $e$ is an idempotent, such that $A$ is not semiprime.   \end{ex}
 
\begin{proof} Let $D$ be a domain of characteristic 0 such that $2\notin 4D$, and $J$ the ideal of $D[X]$ generated by $X^2-X$ and $4X$. Put $D[e] = D[X]/J$ where $e = X+J$.  All elements of $J$ are with 0 constant term; showing that $R\sbq A$.  If $aX\in J$, $a\in D$, then $aX= f(X)(X^2-X) +g(X)4X$. Setting $X=1$ shows $a = g(1)4$. Hence, if $ae =0$, then $a=4d$, some $d\in D$. In particular, $2e\ne 0$. Clearly $D[e]$ is not semiprime. 
\end{proof} 

Notice that this kind of example can never be a Specker algebra (see \cite[Proposition~4.1]{BMMO}). 

If, in Example~\ref{nsp}, $D$ is a UFD and  2 is prime in $D$, then it can be seen that $\B(A) = \{0,1,e, 1-e\}$ and the Pierce stalks of $A$ are $D$ and $D/4D$.  In that case, the kernel of the natural homomorphism of Specker algebra using $B$ the four element boolean algebra,  $D[B]_{\text{Sp}} \to A$ has kernel generated by $4y_e$ (see Definition~\ref{Stdef}). 

As remarked after Proposition~\ref{known}, locally Specker algebras over a semiprime ring $R$ are semiprime. (It suffices that $\phi(R)$ be semiprime, and that is implied by $R$ semiprime.)

In the previous section there are results illustrating the next proposition. 

 \begin{prop} \label{flatE}   Let $R$ be a semiprime ring and $A\in \igr{R}$ with $R\sbq A$. If $A$ satisfies weak going-down with respect to $R$, then $A$ is semiprime.\end{prop}

\begin{proof}  By Theorem~\ref{gdP}, $A_R$ is, in fact, flat. Moreover, in the proof of \ref{gdP}, it is shown that for $\mf{p}\in \spec R$, $A_\mf{p}$ is a Specker algebra over $R_\mf{p}$, by Lemma~\ref{fl}. Hence, $A_\mf{p}$ is a semiprime ring. Then, $A$ embeds in $\prod_{\mf{p}\in \spec R}A_\mf{p}$, a semiprime ring. \end{proof}

\noindent \setcounter{section}{6} \setcounter{thm}{0} \tbf{6. Weak Baer rings and locally Specker algebras.}

The next results are about wB rings, and related rings, and $\igr{R}$.

\begin{thm} \label{epp}  Let $R$ be a wB ring and $R\sbq A=R[E]\in \igr{R}$, $E\sbq \B(A)$, an extension by idempotents that is ring essential over $R$.  Then, $A$ is a wB ring and each Pierce stalk of $A$ is isomorphic to a Pierce stalk of $R$; i.e., $A$ is locally Specker.  \end{thm} 
  
\begin{proof} Note that $A=R[E] = R[\B(A)]$ and that $A$ is semiprime since $R$ is (by essentiality).  This means that $Q(A)$ is a regular ring. Put $T=R[\B(Q(A))]$, $T$ is a wB ring; in fact it is a Baer ring since $Q(A)$ is (Lemma~\ref{Betal}(i)). Recall that $\qcl(R)$ is a regular ring with the same idempotents as $R$. Put $U= \qcl(R)[\B(A)]$, a regular ring. Note that for $u\in U$, $u=\sum_{i=1}^k q_ie_i$, $q_i\in \qcl(R)$, $e_i\in \B(A)$. Then there is a non-zero divisor $r\in R$ with $ur\in A$. By essentiality, $r$ is a non-zero divisor in $A$ and, hence, $r\inv \in \qcl(A)$. This makes $U\sbq \qcl(A)$. However, the regular ring $U$ (Corollary~\ref{clean}(2)), where non-zero divisors are invertible, is ring essential over $A$, as just shown, and a non-zero divisor in $A$ is also one in $U$. Hence, that makes $\qcl(A)\sbq U$. Thus, $\qcl(A) = U$, a regular ring.  
 
 It is next shown that $A$ is wB by showing that the annihilator of an element of $A$ is generated by an idempotent. Before proceeding, note the following: 
 
 ($*$)~if in some semiprime ring $V$ with $v,w\in V$, $\ann_V v=eV$ and $\ann_Vw= fV$, where $vw=0$ and $e,f\in \B(V)$, then, $\ann_V(v+w) =efV$. (If $t\in \ann_V(v+w)$, $tv^2=0$ and $tw^2=0$, giving the result.) This will be used in $A$. 
  
First take $r\in R\sbq A$. Then, $\ann_Rr=fR$, for some $f\in \B(R)$. Let $g\in \B(U)$ with $\ann_U r=gU$.  Then, $fg=f$. If $f\ne g$, then $g(1-f)\ne 0$. There is a non-zero divisor $r'\in R$ with $0\ne g(1-f)r'\in A$ and $a\in A$ with $0\ne g(1-f)r'a\in R$. But, $g(1-f)r'ar=0$ putting $g(1-f)r'a\in fR$, a contradiction unless $g(1-f)=0$ and $f=g$. This means that $\ann_Ar=fA$, since $fU\cap A=fA$. 

The next step is to take $re\in A$, $r\in R$ and $e\in \B(A)$. Then, there is $f\in \B(R)$ with $\ann_Ar= fA$, as above. Consider the idempotent $g= ef+(1-e)f + (1-e)(1-f)$. Then, $gre=0$. Suppose $a\in \ann_Are$. Notice that $are=0$ implies that $ae=aef$. Then, $ag = aef+a(1-e)f +a(1-e)(1-f) = ae +af -ae +a -ae -af +ae = a$. 

Finally, a typical element of $a=\sum_{i=1}^kr_ie_i$, $\{e_1,\ldots, e_k\}$ an orthogonal set from $\B(A)$; for each $i$, put $\ann_Ar_ie_i =g_iA$, $g_i\in \B(A)$. Then, by remark ($*$), $\ann_Aa=(\prod_{i=1}^kg_i)A$. Hence, $A$ is wB.

The last thing to check is the statement about the Pierce stalks of $A$. Pick $y\in \spec \B(A)$ and set $x=y\cap \B(R)$.  Then, $y\in \spec \B(U)$ since $A$ and $U$ have the same idempotents (because $A$ is wB, see \cite[Lemma~3.1]{Be}). Consider the two sequences: (a)~$R_x\twoheadrightarrow A_y \to U_y$, and (b)~$R_x \to (\qcl(R))_x \to U_y$. In (b), the first arrow is monic, and  the second is monic since both $(\qcl(R))_x$ and $U_y$ are fields (\cite[Lemma~3.1]{Be} again). Hence, the first arrow of (a) is also monic and, therefore, an isomorphism. Thus, $A_y\cong R_x$, as required.
 \end{proof}

 \begin{cor} \label{sh}  With the hypotheses of Theorem~\ref{epp}, if $R$ is also semi-hereditary, then $A$ is a locally Specker $R$-algebra, and is semi-hereditary.   \end{cor} 
 
  \begin{proof} Since $A$ is a wB ring (Theorem~\ref{epp}), by \cite[Theorem~4.1]{Be}, it suffices to show that the Pierce stalks of $A$ are Pr\"{u}fer domains.  But, by Theorem~\ref{epp}, the Pierce stalks of $A$ are among those of $R$, and, hence, are Pr\"{u}fer domains. 
 \end{proof}
 
 Note that by \cite[Proposition~3.5]{DeM}, when $R$ is of the form $\cx$, wB and semi-hereditary coincide.

 It will be seen that there are indeed many examples where  Theorem~\ref{epp} applies. 
 
 \begin{thm}\label{Rpp1}   Let $R$ be a semiprime ring, $Q(R)$ its complete ring of quotients and $A = R[\B(Q(R)]$.  Then, $A_R$ is flat if and only if $R$ is a wB ring. When this happens, $A$ is a locally Specker $R$-algebra. \end{thm}

\begin{proof} By construction, $A$ is a Baer ring and is thus semiprime.

Assume first that $A_R$ is flat.  Then, $A_R$ is faithfully flat (Lemma~\ref{intff}) and, being Baer, is an f-ring. Then, by \cite[Remark~2, \S3]{Jp}, $R$ is an f-ring.

It follows that $\qcl(A)\sbq Q(A)=Q(R)$, since $R\sbq A\sbq Q(R)$. (Flatness does not come into play here.)  Since $A$ is a Baer ring, $\qcl(A)$ is regular and flat over $A$. Then using \cite[Theorem~2.1]{C}, as in the proof of Proposition~\ref{comp}, $\minspec A$ is compact. Moreover, $\minspec R$ is a continuous image of $\minspec A$, making $\minspec R$ compact. 

For $r\in R$,  $\ann_R r= (\ann_A r)\cap R =eA \cap R$, for some $e\in \B(A)$.  But then, $\ann_A r = \ann_A (1-e)$. From this $\ann_Rr = (\ann_A (1-e)) \cap R= \ann_R (1-e)$.  Then, by \cite[Lemma~7.7]{BMO}, $\ann_Rr$ is a pure ideal.  Since, $R$ is an f-ring, $\ann_Rr$ is generated by idempotents, making $R$ an awB ring.  Then, using the fact that $\minspec R$ is compact, \cite[\S2, Lemma~$\beta$ and Theorem]{AM} (quoted in \cite[before Lemma~1.4]{NR}), $R$ is a wB ring (note that in \cite{AM}, a wB ring is called a Baer ring, contrary to the usual usage).  

In the other direction, assume that $R$ is a wB ring. The ideas from the proof of Theorem~\ref{epp} will be used. Since $A$ is a subring of the semiprime ring $Q(R)$, it is semiprime. Following the notation of the proof of Theorem~\ref{epp}, $T=R[\B(Q(A))]$ (here $T=A$) and $U=\qcl(R)[\B(A)]$. As in that proof, $U=\qcl(A)$. Then, the argument in the last paragraph of the proof of \ref{epp} shows that $A$ is locally Specker over $R$, and, hence, $A_R$ is flat. 
\end{proof}  

The proof of Theorem~\ref{Rpp1} contains the following, using the references to \cite{NR} and \cite{AM}: \emph{If $R$ is awb and $\qcl(R)$ is regular, then $R$ is wB.}

Combining \ref{epp}, \ref{Rpp1} and their proofs leads to the following. 
 
 \begin{prop} \label{onwB}  The following are equivalent for a semiprime ring $R$: 
 
(1) $R$ is wB.

(2)~Any ring $A=R[E]\in \igr{R}$ is wB where $E\sbq \B(Q(R))$.

(3)~Any ring $A=R[E]\in \igr{R}$ where $E\sbq \B(Q(R))$ is such that $A_R$ is flat.

(4)~There exists $A=R[E]\in \igr{R}$ where $E\sbq \B(Q(R))$ such that $A$ is wB and $A_R$ is flat.

(5)~Any ring $A=R[E]\in \igr{R}$ where $E\sbq \B(Q(R))$ is locally Specker. \end{prop}

\begin{proof} (1) $\Rightarrow$ (2) by Theorem~\ref{epp}. 

(4) $\Rightarrow$ (1): Since $A$ is wB it is an f-ring. The argument in the third paragraph of the proof of Theorem~\ref{Rpp1} shows that $R$ is awB. Since $A$ is wB, $\qcl(A)$ is regular and the argument of the second paragraph of the proof of \ref{Rpp1} shows that $\minspec R$ is compact and, hence, $R$ is wB, as in \ref{Rpp1}. 

(2) $\Rightarrow$ (1) is obvious, as is (1)~$\Rightarrow$~(4). 

(1) $\Rightarrow$ (3) by \ref{epp} since a locally Specker algebra is $R$-flat. 

(3) $\Rightarrow$ (4) by taking $A=R[\B(Q(R))]$, which is Baer and generated by idempotents.  Then, (4) $\Rightarrow$ (1) shows that $R$ is wB.  Then, by \ref{Rpp1}, $A_R$ is flat.

(1) $\Rightarrow$ (5) by \ref{epp}.

(5) $\Rightarrow$ (3) because locally Specker implies $R$-flat. 
\end{proof} 

Theorem~\ref{Rpp1} can be restated when $R=\cx$. 

\begin{cor} \label{bdc}   Let $R = \cx$ for some completely regular topological space $X$, and $A = R[\B(Q(X))]$. Then $A_R$ is flat if and only if $X$ is basically disconnected (i.e., $R$ is a wB ring). \end{cor}

\begin{proof} This follows immediately from Theorem~\ref{Rpp1} and \cite[Proposition~3.5]{DeM} (this latter characterizes rings $\cx$ that are wB). \end{proof} 

The rings $\cx$ give a rich selection of examples.

Recall that a space $X$ is an \emph{almost P-space} (see \cite[proposition~1.1]{Le}) when each non-empty $G_\delta$ set in $X$ has non-empty interior.   A set of equivalent conditions is in \cite[Proposition~1.1]{Le}; some of these will be used in the following. If $X$ is a P-space (\cite[4J]{GJ}, the spaces where $\cx$ is regular) then it is an almost P-space, but \cite{Le} has many kinds of examples of almost P-spaces that are not P-spaces.  However, if $X$ is an almost P-space that is not a P-space, then $\qcl(X)$ is not regular;  in fact $\qcl(X) = \cx$ (which follows readily from \cite[Proposition~1.1(i)]{Le}).

\begin{prop} \label{aps1}   Let $X$ be an almost P-space that is not a P-space. Then, $A=\cx[\B(Q(X))]$ is not flat as an $R$-module.  Moreover, there is such a space $X$ where $\cx$ is an f-ring. \end{prop}

\begin{proof} As in all such rings, $A$ is a Baer ring. Assume that $A_R$ is flat.  Put $B=\qcl(A)$, a regular ring. Then, $B_A$ is flat and, hence, $B_R$ is flat.  But $B$ is module essential over $\cx$, and, hence, is a regular (generalized) ring of quotients of $\cx$. By \cite[Theorem~2.1]{C}, if one regular ring of quotients is $R$-flat, so is any regular ring of quotients. In particular, the regular ring $Q(X)$ is $R$-flat. By \cite[Theorem~3.1]{M}, $\minspec \cx$ is compact. By Proposition~\ref{qclreg}, $\qcl(X)$ is regular (recall that every ring of the form $\cx$ has the a.c.\ by \cite[example (i) following Definition~3.2]{HJ}).   This is a contradiction.  Hence, $A_R$ is not flat. 

For the final statement, the space of \cite[Example~3.2]{NR}  is an example.  It is $X=\beta\N\setminus \N$, an almost P-space (\cite[Page~284, Examples~2]{Le}) that is strongly 0-dimensional; hence, $\cx$ is an f-ring.
\end{proof}

Notice that \cite[Example~3.2]{NR} is an awB ring (hence, also an f-ring) where the equivalent conditions of Proposition~\ref{onwB} fail.  \epar

\end{document}